\documentclass[11pt]{amsart}

\usepackage[margin=1.1in,
           top=1.2in,
           bottom=1.2in,
           footskip=.25in]{geometry}
\usepackage{amssymb}
\usepackage{amscd}
\usepackage[all]{xy}
\usepackage{bbm}
\usepackage{mathrsfs}
\usepackage{enumerate}
\usepackage{tikz}
\usepackage{geometry}
\usepackage{stmaryrd}
\usepackage{etoolbox}
\usepackage{array}
\usepackage{hyperref}

\newcommand{\N}{\mathbb{N}}
\newcommand{\Z}{\mathbb{Z}}

\newcommand{\R}{\mathbb{R}}

\newcommand{\eps}{\varepsilon}

\newcommand{\bd}{\partial}

\newcommand{\supp}{\operatorname{supp}}

\newcommand{\vol}{\operatorname{Vol}}

\newcommand{\area}{\operatorname{Area}}

\newcommand{\hv}{\mathfrak h}
\newcommand{\vc}{\operatorname{VC}}

\newcommand{\B}{\mathcal B}
\newcommand{\C}{\mathcal C}

\newcommand{\vz}{\operatorname{VZ}}

\newcommand{\M}{\mathbf M}

\theoremstyle{plain}
\newtheorem{theorem}{Theorem}[section]

\newtheorem{prop}[theorem]{Proposition}

\newtheorem{conj}[theorem]{Conjecture}

\theoremstyle{definition}
\newtheorem{defn}[theorem]{Definition}

\theoremstyle{remark}
\newtheorem{rem}[theorem]{Remark}

\begin{document}

\title[An Alternative for CMCs]{An alternative for constant mean curvature hypersurfaces}
\author{Liam Mazurowski}
\author{Xin Zhou}
\address{Lehigh University, Department of Mathematics, Bethlehem, Pennsylvania 18015}
\email{lim624@lehigh.edu}
\address{Cornell University, Department of Mathematics, Ithaca, New York 14850}
\email{xinzhou@cornell.edu}

\begin{abstract}
Let $M^{n+1}$ be a closed manifold of dimension $3\le n+1\le 7$ equipped with a generic Riemannian metric $g$. Let $c$ be a positive number.  We show that, either there exist infinitely many distinct closed hypersurfaces with constant mean curvature equal to $c$, or there exist infinitely many distinct closed hypersurfaces with constant mean curvature less than $c$ but enclosing half the volume of $M$.
\end{abstract}

\maketitle


\section{Introduction}
\label{section:introduction}

Let $(M,g)$ be a closed Riemmanian manifold. Let $\mathcal M$ denote the moduli space of all smooth, closed, almost-embedded constant mean curvature hypersurfaces in $M$. Note that we do not fix the topology of elements in $\mathcal M$, and that the value of the mean curvature is allowed to vary between different hypersurfaces in $\mathcal M$. In this short note, we contribute to the study of $\mathcal M$ by proving the following alternative for constant mean curvature hypersurfaces. 

\begin{theorem}
\label{theorem:dichotomy}
Let $M^{n+1}$ be a closed manifold of dimension $3\le n+1\le 7$. Let $g$ be a generic Riemannian metric on $M$. Then, for every constant $c > 0$, at least one of the following two statements is true: 
\begin{itemize}
\item[(i)] There exist infinitely many distinct smooth, closed, almost-embedded hypersurfaces in $M$ with constant mean curvature equal to $c$;
\item[(ii)] There exist infinitely many distinct smooth, closed, almost-embedded hypersurfaces in $M$ with constant mean curvature less than $c$ but enclosing half the volume of $M$. 
\end{itemize}
\end{theorem}

The proof of Theorem \ref{theorem:dichotomy} is based on the authors' recently developed min-max theory \cite{mazurowski2024infinitely} for finding critical points of functionals of the form $\Omega \mapsto \area(\bd \Omega) + f(\vol(\Omega))$. Here $\Omega \subset M$ is an open set with sufficiently nice boundary, and $f$ is an arbitrary smooth function. 

\subsection{Background and Motivation} In this subsection, we briefly survey some of the known results and open problems concerning the space $\mathcal M$. It is known that for a generic metric $g$ on $M$, the moduli space $\mathcal M$ of constant mean curvature hypersurfaces has the structure of a 1-dimensional manifold; see for example \cite{pacard2005foliations}. This manifold may have both compact and non-compact components. 

The finer structure of $\mathcal M$ is best understood in the very small mean curvature and very large mean curvature regimes. Regarding surfaces with large mean curvature, Ye \cite{ye1991foliation} proved that there exist foliations by CMC spheres in small neighborhoods of non-degenerate critical points of the scalar curvature. If the scalar curvature function $R$ is Morse, this implies that $\mathcal M$ contains at least $\mu(M)$ spheres with constant mean curvature $c$ when $c$ is large enough. Here $\mu(M)$ is the Morse number of $M$, i.e., the minimum number of critical points of a Morse function on $M$.  Later, Pacard and Xu \cite{pacard2009constant} proved that even when $R$ is not Morse, $M$ still contains at least $\operatorname{cat}(M)$ spheres with constant mean curvature $c$ when $c$ is large enough. Here $\operatorname{cat}(M)$ is the Lusternik-Schnirelmann category of $M$. 

In general, these CMC spheres do not exhaust the high mean curvature portion of $\mathcal M$. For example, Mazzeo and Pacard \cite{pacard2005foliations} have constructed CMCs with large mean curvature by perturbing small tubes around geodesics. We still do not have a complete understanding of the space of all CMC hypersurfaces with large mean curvature. In fact, even if one restricts the topology, there is not yet a complete description of the space of CMC spheres with large mean curvature. For instance, Pacard claims in joint work with Malchiodi \cite{pacard2005constant} to be able to construct CMCs which look like a connect sum of two small spheres joined by a Delaunay tube.

In the small mean curvature regime, constant mean curvature surfaces can be found by perturbing minimal surfaces. CMCs with small mean curvature can also be constructed by doubling minimal surfaces; see \cite{Mazurowski22} for a min-max construction of such doublings. Dey \cite{Dey23} proved that the number of CMC hypersurfaces with mean curvature $c$ goes to infinity at a quantitative rate as $c\to 0$. Again a complete description of the space of CMC hypersurfaces with small mean curvature still seems out of reach. 

For general values of $c > 0$, the second named author and J. Zhu \cite{zhou2019min} proved that $M$ always contains at least one almost-embedded hypersurface with constant mean curvature $c$. It is conjectured that there should always be a second almost-embedded hypersurfaces with constant mean curvature $c$. 

\begin{conj}[Twin Bubble Conjecture \cite{Zhou-ICM22}]
    Let $(M^{n+1},g)$ be a closed Riemannian manifold of dimension $3\le n+1 \le 7$. Then, for every constant $c > 0$, there are at least two distinct closed, almost-embedded hypersurfaces with constant mean curvature $c$ in $M$. 
\end{conj}

This conjecture is motivated in part by a famous conjecture of Arnold on the existence of closed magnetic geodesics in Riemannian 2-spheres \cite{arnold2004arnold}. Arnold's conjecture states that every Riemannian 2-sphere $(S^2,g)$ should contain at least two closed curves with constant geodesic curvature $\kappa$ for every constant $\kappa > 0$. Schneider \cite{schneider2011closed} has proved Arnold's conjecture when the metric satisfies certain positive curvature conditions, but the full conjecture remains open. 

\begin{rem} It is not entirely clear that the number two is optimal in the twin bubble conjecture. In fact, to the authors' knowledge, it is not known whether there exist a closed manifold $M$ and a constant $c >0$ such that $M$ contains only finitely many closed hypersurfaces with constant mean curvature $c$. 
\end{rem} 

The present work grew out of the authors' efforts to understand the Twin Bubble Conjecture. Theorem \ref{theorem:dichotomy} shows that, at least for generic metrics, if the Twin Bubble Conjecture fails for a particular constant $c >0$, then $M$ must contain infinitely many constant mean curvature hypersurfaces with mean curvature less than $c$, each enclosing half the volume of $M$. In particular, this suggests that a thorough understanding of the space of half-volume constant mean curvature hypersurfaces in $M$ may be helpful for constructing hypersurfaces with prescribed constant mean curvature. 
The work of the authors \cite{mazurowski2023half}\cite{mazurowski2024infinitely} shows that a closed manifold $M^{n+1}$ of dimension $3\le n+1\le 5$ equipped with a generic metric $g$ always contains a sequence of closed, almost-embedded half-volume constant mean curvature hypersurfaces $\Sigma_p$ with $\area(\Sigma_p)\to \infty$ as $p\to \infty$.  In view of Theorem \ref{theorem:dichotomy}, it would be very interesting to determine whether the mean curvature of these hypersurfaces goes to 0 as $p\to \infty$. 

\subsection{Outline of Paper} In Section \ref{section:preliminaries}, we first discuss some preliminary material from geometric measure theory. We then recall the authors' min-max theory for functionals of the form $\Omega\mapsto \area(\bd \Omega) + f(\vol(\Omega))$. Finally, we briefly review the idea of the suspension of a topological space. 

In Section \ref{section:dichotomy}, we give the proof of Theorem \ref{theorem:dichotomy}.  The proof is divided into three steps. In the first step, we apply Dey's lifting construction \cite{Dey23} to find infinitely many relative homotopy classes suitable for applying min-max theory to the functional 
\[
\Omega\mapsto \area(\bd \Omega) + c\left\vert \vol(\Omega) - \frac{1}{2}\vol(M)\right\vert.
\]
In the second step, we regularize this functional to 
\[
F_{\eps,\delta}(\Omega) = \area(\bd \Omega) - \delta \int_\Omega h + c \sqrt{\left(\vol(\Omega)-\frac{1}{2}\vol(M)\right)^2 + \eps^2}, 
\]
where $h\colon M \to \R$ is a suitable Morse function and $\eps$ and $\delta$ are small positive constants. We then apply min-max theory to find smooth critical points $\Omega_{\eps,\delta}$ of $F_{\eps,\delta}$. Finally, in the third step we let $\delta \to 0$ and then let $\eps \to 0$ and obtain convergence $\Omega_{\eps,\delta}\to \Omega$. We then show that $\Sigma = \bd \Omega$ is either an almost-embedded hypersurface with constant mean curvature $c$, or an almost-embedded hypersurface with constant mean curvature less than $c$ but enclosing half the volume of $M$.

\subsection{Acknowledgements}
L.M. acknowledges the support of an AMS Simons Travel Grant. X.Z. acknowledges the support by NSF grant DMS-1945178, and a grant from the Simons Foundation (1026523, XZ).

\section{Preliminaries}
\label{section:preliminaries}

In this section we introduce some background material that will be needed for the proof of the main theorem. 

\subsection{Geometric Measure Theory}
In this section, we recall some background from geometric measure theory. For more details we refer to \cite{almgren1965theory}, \cite[Chapter 2]{pitts2014existence}, and \cite{simon1983lectures}. Let $M^{n+1}$ be a closed manifold. 
\begin{itemize}
\item Let $\mathcal V(M)$ denote the space of $n$-dimensional varifolds in $M$. 
\item Given $V\in \mathcal V(M)$, let $\|V\|$ denote the weight measure associated to $V$. 
\item Let $\mathcal Z_n(M,\Z_2)$ be the space of $n$-dimensional flat chains mod 2 in $M$. 
\item Given $T\in \mathcal Z_n(M,\Z_2)$, let $\vert T\vert\in \mathcal V(M)$ denote the varifold induced by $T$. 
\item Let $\mathcal C(M)$ be the space of Caccioppoli sets in $M$. 
\item Given $\Omega \in \C(M)$, let $\bd \Omega\in \mathcal Z_n(M,\Z_2)$ be the flat chain induced by the boundary of $\Omega$. 
\item Let $\mathcal B(M,\Z_2)$ be the set of all $T \in \mathcal Z_n(M,\Z_2)$ such that $T = \bd \Omega$ for some $\Omega\in \C(M)$. This is the connected component of the zero cycle in $\mathcal Z_n(M,\Z_2)$.
\item Let $\mathcal F$, $\mathbf F$, and $\mathbf M$ denote the flat topology, the $\mathbf F$ topology, and the mass topology, respectively. By definition, the $\mathbf F$ topology on $\C(M)$ is given by 
\[
\mathbf F(\Omega_1,\Omega_2) = \mathcal F(\Omega_1,\Omega_2) + \mathbf F(\vert \bd \Omega_1\vert, \vert \bd \Omega_2\vert). 
\]
\item Let $\vz(M,\Z_2)$ and $\vc(M)$ denote the $\vz$ and $\vc$ spaces. The $\vz$ space was first introduced by Almgren \cite{almgren1965theory}, and the related $\vc$ space first appeared in \cite{wang2023existence}.
\end{itemize}
We will describe the $\vz$ and $\vc$ spaces in slightly more detail, since this spaces are less well-known.   For more a more thorough introduction to these spaces and their applications in min-max theory see \cite{wang2023existence} and \cite{mazurowski2024infinitely}.

\begin{defn}
    The space $\vz(M,\Z_2)$ consists of all pairs $(V,T)\in \mathcal V(M)\times \mathcal B(M,\Z_2)$ such that there exists a sequence $T_i \in \mathcal B(M,\Z_2)$ with $\vert T_i\vert \to V$ as varifolds and $T_i \to T$ in the flat topology on $\mathcal B(M,\Z_2)$. 
\end{defn}

Given $(V,T)\in \vz(M,\Z_2)$ it is not necessarily true that $V = \vert T\vert$. However, it is always true that $\|\, \vert T\vert\, \| \le \|V\|$ as measures. We equip $\vz(M,\Z_2)$ with the $\mathscr F$ topology given by 
\[
\mathscr F((V_1,T_1),(V_2,T_2)) = \mathbf F(V_1,V_2) + \mathcal F(T_1,T_2).
\]
The $\vc$ space is defined similarly but using $\C(M)$ in place of $\B(M,\Z_2)$. 

\begin{defn}
    The space $\vc(M)$ consists of all pairs $(V,\Omega) \in \mathcal V(M)\times \C(M)$ such that there exists a sequence $\Omega_i \in \C(M)$ with $\vert \bd \Omega_i\vert \to V$ as varifolds and $\Omega_i\to \Omega$ as Caccioppoli sets. The $\mathscr F$ topology on $\vc(M)$ is given by $\mathscr F((V_1,\Omega_1),(V_2,\Omega_2)) = \mathbf F(V_1,V_2) + \mathcal F(\Omega_1,\Omega_2)$. 
\end{defn}

\subsection{Min-Max Notions} In this subsection, we describe the min-max constructions for the $E$ and $F$ functionals developed in \cite{mazurowski2024infinitely}. Let $M$ be a closed Riemannian manifold. 

First we describe the relevant parameter spaces and homotopy classes. Let $I(1,k)$ denote the cubical complex structure on $[0,1]$ with vertices $[0],[3^{-k}],[2\cdot 3^{-k}],\hdots,[1]$ and edges $[0,3^{-k}], [3^{-k},2\cdot 3^{-k}], \hdots, [1-3^{-k},1]$. Then let 
\[
I(m,k) = \underbrace{I(1,k)\otimes \cdots \otimes I(1,k)}_{m\text{ times}}
\]
be the cubical complex structure subdividing $I^m$ into $3^{mk}$ congruent subcubes. 

Let $X$ be a cubical subcomplex of $I(m,k)$ for some $m,k\in \N$.  Fix an $\mathbf F$-continuous map $\Phi\colon X\to (\mathcal B(M,\Z_2),\mathbf F)$. 

\begin{defn}
\label{abs-homotopy}
    The $X$-homotopy class of the map $\Phi$ is the set of all sequences of maps $\{\Phi_i\colon X\to (\mathcal B(M,\Z_2),\mathbf F)\}$ such that there exist $\mathbf F$ continuous homotopies $H_i\colon X\times [0,1]\to (\mathcal B(M,\Z_2),\textcolor{cyan}{\mathbf F})$ satisfying $H_i(x,0) = \Phi(x)$ and $H_i(x,1) = \Phi_i(x)$.
\end{defn}

Now assume that $Z$ is a cubical subcomplex of $X$. Consider an $\mathbf F$-continuous map $\Psi\colon X\to (\C(M),\mathbf F)$.

\begin{defn}
\label{rel-homotopy}
    The $(X,Z)$-relative homotopy class of the map $\Psi$ is the set of all sequences of maps $\{\Psi_i\colon X\to (\C(M),\mathbf F)\}$ such that there exist $\mathbf F$-continuous homotopy maps $H_i\colon X\times [0,1]\to (\C(M),\mathbf F)$ satisfying $H_i(x,0) = \Psi(x)$ and $H_i(x,1) = \Psi_i(x)$ and 
    \[
    \limsup_{i\to\infty} \left[\sup_{(z,t)\in Z\times [0,1]} \mathbf F(H_i(z,t),\Psi(z))\right] = 0. 
    \]
\end{defn}

\begin{rem}
    In \cite{mazurowski2024infinitely}, the min-max theory was developed only requiring the homotopies $H_i$ in Definitions \ref{abs-homotopy} and \ref{rel-homotopy} to be $\mathcal F$ continuous. However, inspecting the proofs reveals that it is actually possible to require that $H_i$ is $\mathbf F$ continuous. Indeed, Propositions 1.14 and 1.15 in \cite{zhou2020multiplicity} can be used to ensure all relevant homotopies are $\mathbf F$-continuous.  See the remark above Definition 3.4 in \cite{dey2022comparison}, where the same observation was used.
\end{rem}

Next, we describe the relevant functionals. Fix a smooth function $f\colon [0,\vol(M)] \to \R$. Let $\hv = \frac{1}{2}\vol(M)$ and assume $f$ is even in the sense that $f(\hv + t) = f(\hv - t)$ for all $t\in [0,\hv]$. 

\begin{defn}
    Define $E \colon \mathcal B(M,\Z_2) \to \R$ by 
    \[
    E(T) = \M(T) + f(\vol(\Omega))
    \]
    where $\Omega\in \C(M)$ is any set satisfying $\bd \Omega = T$. Note that $E$ is well-defined since $f$ is assumed to be even. The $E$ functional can also be extended to $\vz(M,\Z_2)$ by setting $E(V,T) = \|V\|(M) + f(\vol(\Omega))$ where $\bd \Omega = T$. 
\end{defn}

\begin{defn}
    Let $\Pi$ be the $X$-homotopy class of a map $\Phi\colon X\to (\mathcal B(M,\Z_2),\mathbf F)$. We define the min-max value 
    \[
    L^E(\Pi) = \inf_{\{\Phi_i\}\in \Pi} \left[\limsup_{i\to \infty} \sup_{x\in X} E(\Phi_i(x))\right].
    \]
\end{defn}

In \cite{mazurowski2024infinitely}, the authors proved the following min-max theorem for the $E$ functional. 

\begin{theorem}\label{E-min-max} 
    Let $M^{n+1}$ be a closed manifold of dimension $3\le n+1\le 7$. Let $\Pi$ be the $X$-homotopy class of a map $\Phi\colon X\to (\mathcal B(M,\Z_2),\mathbf F)$. Assume that $L^E(\Pi) > 0$. Then there exists $(V,T)\in \vz(M,\Z_2)$ with $E(V,T) = L^E(\Pi)$. Moreover, $(V,T)$ has the following regularity. Choose $\Omega\in \C(M)$ with $\bd \Omega = T$ and then let $H = -f'(\vol(\Omega))$. 
    \begin{itemize}
        \item[(i)] If $H\neq 0$ then there exists a smooth, closed, almost-embedded hypersurface $\Lambda$ with constant mean curvature $H$ such that $\bd \Omega = \Lambda$. Moreover, there is a (possibly empty) collection of minimal hypersurfaces $\Sigma_1,\hdots,\Sigma_k$ together with multiplicities $m_1,\hdots,m_k \in \N$ such that 
        \[
        V = \vert \Lambda \vert + \sum_{j=k}^\ell m_j \vert \Sigma_j\vert. 
        \]
        The hypersurfaces $\Lambda,\Sigma_1,\hdots,\Sigma_k$ are all disjoint.
        \item[(ii)] If $H = 0$ then there exists a collection of minimal hypersurfaces $\Lambda_1,\hdots,\Lambda_q$ such that $\bd \Omega = \Lambda_1 \cup \hdots \cup \Lambda_q$. Moreover, there is a (possibly empty) collection of minimal hypersurfaces $\Sigma_1,\hdots,\Sigma_k$ and multiplicites $\ell_1,\hdots,\ell_q,m_1,\hdots,m_k \in \N$ such that 
        \[
        V = \sum_{i=1}^q \ell_i \vert \Lambda_i\vert + \sum_{j=1}^k m_j\vert \Sigma_j\vert.
        \]
        The hypersurfaces $\Lambda_1,\hdots,\Lambda_q,\Sigma_1,\hdots,\Sigma_k$ are all disjoint. 
    \end{itemize}
\end{theorem}

Finally, fix a smooth Morse function $h\colon M\to \R$. We assume that $h$ satisfies property (T); see \cite[Definition 2.10]{mazurowski2024infinitely}. The set of Morse functions satisfying property (T) is dense in $C^\infty(M)$ by \cite[Proposition A.3]{mazurowski2024infinitely}.

\begin{defn}
    Define $F\colon \C(M)\to \R$ by 
    \[
    F(\Omega) = \M(\bd \Omega) - \int_\Omega h + f(\vol(\Omega)).
    \]
    The $F$ functional can also be extended to $\vc(M)$ by setting $F(V,\Omega) = \|V\|(M)-\int_\Omega h + f(\vol(\Omega))$. 
\end{defn}

\begin{defn}
    Let $\Pi$ be the $(X,Z)$-homotopy class of a map $\Psi\colon X\to (\C(M),\mathbf F)$. We define the min-max value 
    \[
    L^F(\Pi) = \inf_{\{\Psi_i\}\in \Pi} \left[\limsup_{i\to\infty} \sup_{x\in X} F(\Psi_i(x))\right].
    \]
\end{defn}

In \cite{mazurowski2024infinitely}, the authors proved the following min-max theorem for the $F$ functional. 

\begin{theorem}
    Let $M^{n+1}$ be a closed manifold of dimension $3\le n+1\le 7$. Let $\Pi$ be the $(X,Z)$-homotopy class of a map $\Psi\colon X\to (\C(M),\mathbf F)$. Assume that 
    \[
    L^F(\Pi) > \sup_{z\in Z} F(\Psi(z)). 
    \]
    Then there exists $(V,\Omega)\in \vc(M)$ satisfying $F(V,\Omega) = L^F(\Pi)$. Moreover, there exists a smooth, almost-embedded hypersurface $\Lambda$ with mean curvature $H = h|_\Lambda - f'(\vol(\Omega))$ such that $\bd \Omega = \Lambda$ and $V = \vert \Lambda\vert$. 
\end{theorem}

Finally, we give the definition of $p$-sweepouts and the volume spectrum.  It is known that the cohomology ring of $\mathcal B(M,\Z_2)$ with $\Z_2$ coefficients is isomorphic to $\Z_2[\lambda]$ with generator $\lambda$ in degree one. This follows from a more general result of Almgren \cite{almgren1962homotopy}; also see Marques-Neves \cite{marques2021morse} for a simpler proof in this special case.

\begin{defn}
Let $X$ be a cubical subcomplex of some $I(m,k)$. A flat continuous map $\Phi\colon X\to \B(M,\Z_2)$ is called a $p$-sweepout provided $\Phi^*\lambda^p \neq 0$ in $H^p(X,\Z_2)$. 
\end{defn}

\begin{defn}
Following \cite{marques2017existence}, we say that a $p$-sweepout $\Phi\colon X\to \B(M,\Z_2)$ has no concentration of mass provided
\[
\lim_{r\to 0} \bigg[\sup\{\M(\Phi(x)\llcorner B(q,r)): x\in X,\ q\in M\}\bigg] = 0.
\]
Let $\mathcal P_p$ denote the set of all $p$-sweepouts with no concentration of mass.
\end{defn}

\begin{defn}
The volume spectrum $\{\omega_p\}_{p\in \N}$ is defined by 
\[
\omega_p = \inf_{\Phi\in \mathcal P_p} \left[\sup_{x\in{\operatorname{dom(\Phi)}}} \M(\Phi(x))\right];
\]
see \cite{gromov2006dimension} and \cite{liokumovich2018weyl}. 
\end{defn}

\subsection{The Suspension of a Topological Space} In this short subsection, we briefly recall the idea of the suspension of a topological space. Later we will use suspensions to construct suitable parameter spaces on which to apply min-max theory.  

\begin{defn} Let $Y$ be a topological space. The suspension of $Y$ is the topological space 
\[
SY = Y\times [-1,1]/\sim
\]
where $(y_1,1)\sim (y_2,1)$ and $(y_1,-1)\sim (y_2,-1)$ for all $y_1,y_2\in Y$. 
\end{defn}

Note that $Y$ is naturally identified with the subspace $Y\times \{0\} \subset SY$. The suspension $SY$ can also be viewed as the union of two cones on $Y$, glued along a common copy of $Y$. More precisely, let $C_+ = Y\times [0,1]/\sim$ where $(y_1,1)\sim (y_2,1)$, and let $C_- = Y\times[-1,0]/\sim$ where $(y_1,-1)\sim(y_2,-1)$. Then $C_+$ and $C_-$ are both cones on $Y$, and $SY$ is obtained by gluing $C_+$ to $C_-$ along $Y\times \{0\}$. Note that if $Y$ is a cubical complex, then $SY$ also has the structure of a cubical complex. 

Finally, consider the case where $Y$ admits a free $\Z_2$ action $y\mapsto -y$. Then this free $\Z_2$ action extends to a free $\Z_2$ action on $SY$ given by $(y,t)\mapsto (-y,-t)$. Let $\overline Y = Y/\Z_2$ and $\overline{SY} = SY/\Z_2$ denote the quotient spaces. Then $\overline Y$ is naturally identified with the subset $(Y\times \{0\})/\Z_2 \subset \overline{SY}$.

\section{Proof of the Alternative}
\label{section:dichotomy}

The goal of this section is to prove Theorem \ref{theorem:dichotomy}. Let $M^{n+1}$ be a closed Riemannian manifold of dimension $3\le n+1\le 7$. Let $g$ be a Riemannian metric on $M$ such that no collection of minimal hypersurfaces encloses half the volume of $M$. Such metrics are generic in the sense of Baire category by \cite[Proposition B.1]{mazurowski2024infinitely}. Fix a constant $c > 0$. 

\begin{defn} Given $\epsilon>0$, let 
$f_\eps(t) = \sqrt{t^2+\eps^2}.$
\end{defn} 

Observe that $\vert f_\eps'(t)\vert \le 1$ for all $t\in \R$. Moreover, the functions $f_\eps$ converge uniformly to the absolute value function as $\eps \to 0$.

\begin{defn}
Let $\hv = \frac{1}{2}\vol(M)$. 
\end{defn}

Next we define the functionals to which we will apply mix-max theory. 

\begin{defn} For each $\eps > 0$, define $E_\eps: \B(M,\Z_2)\to \R$ by
\[ E_{\eps}(T) = \M(T) - c \cdot f_\eps\big(\vol(\Omega)-\hv\big), \]
where $\Omega \in \C(M)$ satisfies $\bd \Omega = T$. Also define 
\[
E_0(T) = \M(T) - c \vert \vol(\Omega)-\hv\vert,
\]
where again $\Omega\in \C(M)$ satisfies $\bd \Omega = T$. 
\end{defn}

Recall that $\{\omega_p\}_{p\in \N}$ denotes the volume spectrum of $M$ and that $\mathcal P_p$ is the set of all $p$-sweepouts with no concentration of mass.
In our case, we our interested in the following alternative min-max values where the area functional is replaced by $E_0$. 

\begin{defn} Define the min-max values 
\[
\alpha_p = \inf_{\Phi \in \mathcal P_p} \left[\sup_{x\in \operatorname{dom}(\Phi)} E_0(\Phi(x)) \right].
\]
\end{defn}

Observe that 
$
\omega_p - c \hv \le \alpha_p \le \omega_p
$
for all $p\in \N$. In particular, we have $\alpha_p \to \infty$ as $p\to \infty$ and so there are infinitely many $p$'s for which there is a gap $\alpha_p < \alpha_{p+1}$. Therefore, the following Theorem \ref{theorem:gap} implies Theorem \ref{theorem:dichotomy} as an immediate corollary. The remainder of this section will be devoted to the proof of Theorem \ref{theorem:gap}.

\begin{theorem}
\label{theorem:gap}
    Fix $p\in \N$ for which there is a gap $\alpha_p < \alpha_{p+1}$. Then there exists a closed, almost-embedded hypersurface $\Lambda$ in $M$ with $\area(\Lambda) \ge \alpha_{p+1}$ such that either 
    \begin{itemize}
        \item[(i)] $\Lambda$ has constant mean curvature $c$, or 
        \item[(ii)] $\Lambda$ has constant mean curvature less than $c$ and encloses half the volume of $M$. 
    \end{itemize}
\end{theorem}

\begin{proof}  
Fix some $p\in \N$ for which $\alpha_p < \alpha_{p+1}$. 

\vspace{0.5em}
{\bf Step 1:} We will construct a suitable relative homotopy class. For this step, we will closely follow the arguments of Dey \cite{Dey23}. 
Choose a $p$-sweepout $\overline \Phi \colon \overline Y\to \B(M,\Z_2)$ with no concentration of mass for which 
\[
\sup_{\bar y\in \overline Y} E_0(\overline \Phi(\bar y)) < \alpha_{p+1}. 
\]
Now let $Y$ be the double cover of $\overline Y$ corresponding to the cohomology class $\overline \Phi^*\lambda$, where $\lambda$ is the non-zero element in $H^1(\B(M,\Z_2),\Z_2)$. Let $y\mapsto -y$ denote the deck transformation of $Y$. Let $\Phi\colon Y\to \C(M)$ be the lift of $\overline \Phi$.  Let $X=SY$ be the suspension of $Y$. Note that $Y$ is naturally identified with the subspace $Y\times \{0\}$ in $X$. Moreover, $X$ carries a free $\Z_2$ action $x \mapsto -x$ extending the deck transformation on $Y$. Let $Z = Y\times [-2^{-1},2^{-1}]\subset X$. Also let $\overline X$ be the quotient of $X$ by the $\Z_2$ action, and let $\overline Z\subset \overline X$ be the quotient of $Z$ by the $\Z_2$ action. 

Next, as in \cite[Section 4, Part 2]{Dey23}, choose a map $\Lambda \colon [0,1] \to (\C(M),\mathcal F)$ with no concentration of mass satisfying $\Lambda(0) = M$ and $\Lambda(1) = \emptyset$. Then define $\Psi'\colon X\to (\C(M),\mathcal F)$ by 
\[
\Psi'(y,t) = \begin{cases}
    \Phi(y), &\text{if } \vert t\vert \le 2^{-1}\\
    \Phi(y) \cap \Lambda(2t-1)  &\text{if } t\ge 2^{-1},\\
    M\setminus (\Phi(-y) \cap \Lambda(-1-2t)), &\text{if } t\le -2^{-1}. 
\end{cases}
\]
The map $\Psi'$ is flat continuous by \cite[Claim 4.2]{Dey23}. Moreover, $\Psi'$ satisfies 
\[
\Psi'(x) = M\setminus \Psi'(-x)
\]
by construction, and therefore $\Psi'$ induces a map $\overline \Psi'\colon \overline X \to \mathcal (B(M,\Z_2),\mathcal F)$. The map $\overline \Psi'$ has no concentration of mass by \cite[Claim 4.3]{Dey23}. 
Hence, by applying discretization followed by interpolation \cite[Section 1.3]{zhou2020multiplicity}, we can replace $\overline \Psi'$ by an $\mathbf F$-continuous map $\overline \Psi\colon \overline X\to (\B(M,\Z_2),\mathbf F)$ with the property that 
\begin{equation}
\label{eq:z-upper-bound}
\sup_{\bar z\in \overline Z} E_0(\overline \Psi(\bar z)) < \alpha_{p+1}. 
\end{equation}
The map $\overline \Psi$ is a $(p+1)$-sweepout by \cite[Section 4, Part 3]{Dey23}.  Finally let $\Psi\colon X\to (\C(M),\mathbf F)$ be the lift of $\overline \Psi$. Let $\Pi$ be the $(X,Z)$ relative homotopy class of the map $\Psi$. 

\vspace{0.5em}

{\bf Step 2:}
The next step is to apply min-max theory in the homotopy class $\Pi$. Fix a smooth Morse function $h\colon M\to \R$ satisfying property (T). For each $\delta > 0$, define $F_{\eps,\delta}\colon \C(M)\to \R$ by 
\[
F_{\eps,\delta}(\Omega) = \M(\bd \Omega) - \delta \int_\Omega h - c\cdot f_\eps(\vol(\Omega) - \hv). 
\]
Also define $\widetilde E_0\colon \C(M)\to \R$ by 
$
\widetilde E_0(\Omega) = \M(\bd \Omega) - c \vert \vol(\Omega) - \hv\vert. 
$

\begin{prop}
\label{prop:quotient}
For any sequence $\{\Psi_i\colon X\to (\C(M),\mathbf F)\}_{i\in \N}\in \Pi$, we have 
\[
\limsup_{i\to \infty}\left[\sup_{x\in X} \widetilde E_0(\Psi_i(x)) \right] \ge \alpha_{p+1}.
\]
\end{prop}

\begin{proof}
The proof is essentially the same as \cite[Section 4, Part 4]{Dey23}. 
Recall that 
\[
\limsup_{i\to \infty}\left[\sup_{z\in Z} \mathbf F(\Psi_i(z),\Psi(z))\right] = 0. 
\]
Moreover, by definition there are $\mathbf F$-continuous homotopies $H_i\colon Z\times [0,1]\to (\C(M),\mathbf F)$ satisfying $H_i(z,0) = \Psi(z)$ and $H_i(z,1) = \Psi_i(z)$ and 
\begin{equation}
\label{eq:z-homotopy}
\limsup_{i\to \infty} \left[\sup_{(z,t)\in Z\times [0,1]} \mathbf F(H_i(z,t),\Psi(z))\right] = 0. 
\end{equation}

Let $C_+$ be the subset of $X$ consisting of all pairs $(y,t)$ with $t \ge 0$. Now define $\Xi_i\colon C_+\to \C(M)$ by 
\[
\Xi_i(y,t) = \begin{cases}
    \Psi_i(y,t), &\text{if } 2^{-1}\le t \le 1\\
    H_i((y,t),2t), &\text{if } 0 \le t \le 2^{-1}. 
\end{cases}
\]
Observe that $\Xi_i$ is $\mathbf F$-continuous and that $\Xi_i(y,0) = \Psi(y,0)$ for all $y\in Y$. Since $\Psi(y,0) = M\setminus \Psi(-y,0)$ for all $y\in Y$, it follows that $\Xi_i$ induces an $\mathbf F$-continuous map $\overline \Xi_i\colon \overline X\to 
(B(M,\Z_2),\mathbf F)$. 
Note that $\overline \Xi_i$ is homotopic to $\overline \Psi$ and therefore that $\overline \Xi_i$ is a $(p+1)$-sweepout. Moreover, $\overline \Xi_i$ is $\mathbf F$-continuous and hence has no concentration of mass. Therefore, it follows that 
\[
\sup_{\bar x \in \overline X}  E_0(\overline \Xi_i(\bar x)) \ge \alpha_{p+1}. 
\]
Next, observe that (\ref{eq:z-upper-bound}) and (\ref{eq:z-homotopy}) imply that 
\[
\sup_{(z,t)\in Z\times [0,1]} \widetilde E_0(H_i(z,t)) < \alpha_{p+1}
\]
provided $i$ is sufficiently large. 
Together with the previous inequality, this implies that 
\[
\sup_{x\in X} \widetilde E_0(\Psi_i(x)) \ge \alpha_{p+1}
\]
for all sufficiently large $i$. This proves the proposition. 
\end{proof}
 
Consider any sequence $\{\Psi_i\colon X\to (\C(M),\mathbf F)\}\in \Pi$. It follows from Proposition \ref{prop:quotient} that 
\[
L^{F_{\eps,\delta}}(\{\Psi_i\}) \ge \alpha_{p+1} - \delta \vol(M) \cdot \sup_M \vert h\vert - c \cdot \sup_{t\in \R} \big\vert f_\eps(t) - \vert t\vert \big\vert.
\]
Likewise, for all $z\in Z$ we have 
\[
F_{\eps,\delta}(\Psi(z)) < \sup_{\bar z\in \overline Z} E_0(\overline \Psi(\bar z)) + \delta \vol(M)\cdot \sup_M \vert h\vert + c \cdot \sup_{t\in \R} \big \vert f_\eps(t) - \vert t\vert\big \vert.  
\]
Since $f_\eps$ converges uniformly to the absolute value function as $\eps \to 0$, it therefore follows that 
\[
L^{F_{\eps,\delta}}(\Pi) > \sup_{z\in Z} F_{\eps,\delta}(\Psi(z))
\]
provided $\eps$ and $\delta$ are small enough. 

Consequently, if $\eps$ and $\delta$ are sufficiently small, we can apply min-max theory for $F_{\eps,\delta}$ in the homotopy class $\Pi$; see Theorem 1.10 in \cite{mazurowski2024infinitely}. It follows that there exists $(V_{\eps,\delta},\Omega_{\eps,\delta})\in \vc(M)$ which is stationary for $F_{\eps,\delta}$ and satisfies $F_{\eps,\delta}(V_{\eps,\delta},\Omega_{\eps,\delta})= L^{F_{\eps,\delta}}(\Pi)$. Moreover, there is a smooth, almost-embedded hypersurface $\Sigma_{\eps,\delta}$ with $\bd \Omega_{\eps,\delta} = \Sigma_{\eps,\delta}$ and $V_{\eps,\delta} = \vert \Sigma_{\eps,\delta}\vert$. The mean curvature of $\Sigma_{\eps,\delta}$ is equal to $\delta h|_{\Sigma_{\epsilon, \delta}} + h_{\eps,\delta}$ where $h_{\eps,\delta} = -c \cdot f_\eps'(\vol(\Omega_{\eps,\delta})-\hv)$ is a constant satisfying $\vert h_{\eps,\delta}\vert \le c$. Finally, $(V_{\eps,\delta},\Omega_{\eps,\delta})$ satisfies property (R)\footnote{Property (R) asserts that there is a uniform number $N$ such that $(V_{\eps,\delta},\Omega_{\eps,\delta})$ must be almost minimizing for $F_{\eps,\delta}$ in at least one annulus in any collection of $N$ concentric annuli; c.f. \cite[Proposition 3.15]{mazurowski2024infinitely}. This property is needed to apply the compactness theory developed in \cite[Section 5]{mazurowski2024infinitely}}. 
for $F_{\eps,\delta}$. 

\vspace{0.5em}

{\bf Step 3:} In the final step, we will pass to the limit $\delta \to 0$ and then the limit $\eps \to 0$ to construct the desired hypersurfaces.
First, fix a small $\eps > 0$ and consider a sequence $\delta_j \to 0$. For notational convenience, let $V_{\eps,j} = V_{\eps,\delta_j}$ and let $\Omega_{\eps,j} = \Omega_{\eps,\delta_j}$. Passing to a limit along a subsequence, we can obtain convergence $V_{\eps_j}\to V_\eps$ as varifolds and $\bd \Omega_{\eps_j} \to \Omega_\eps$ as Caccioppoli sets. The pair $(V_\eps,\bd \Omega_\eps)$ is stationary for $E_\eps$. Since the first variation of $(V_{\eps,j},\Omega_{\eps,j})$ is uniformly bounded, we can apply the compactness theory in \cite[Section 5]{mazurowski2024infinitely} to deduce that one of the following alternatives holds:
\begin{itemize}
    \item[(i)] There is a smooth, closed, almost-embedded hypersurface $\Lambda_\eps$ with non-zero constant mean curvature such that $\bd \Omega_\eps = \Lambda_\eps$ and $V_\eps = \vert \Lambda_\eps\vert$. 
    \item[(ii)] The varifold $V_\eps$ is induced by a collection of disjoint, closed minimal hypersurfaces with integer multiplicities. Moreover, some subcollection of the minimal hypersurfaces bounds the region $\Omega_\eps$. 
\end{itemize}
Finally, note that the fact that $(V_\eps, \Omega_\eps)$ is stationary for $E_\eps$ implies that the mean curvature of $\supp \|V_\eps\|$ satisfies $\vert H_\eps\vert = c \vert f_\eps'(\vol(\Omega_\eps)-\hv)\vert$. 

Next, we aim to take a limit of $(V_\eps,\Omega_\eps)$ as $\eps \to 0$. Select a sequence $\eps_k \to 0$. For notational convenience, let $V_k = V_{\eps_k}$ and let $\Omega_k = \Omega_{\eps_k}$.  Passing to a limit along a subsequence, we can obtain convergence $V_k\to V$ as varifolds and $\Omega_k \to \Omega$ as Caccioppoli sets. Let $H_k$ be the mean curvature of $\supp \|V_k\|$. Passinng to a futher subsequence if necessary, we may assume that $H_k\to H$ as $k\to \infty$. Note that $\vert H\vert \le c$. Since $H_k$ is uniformly bounded, we can again appeal to the compactness theory in \cite[Section 5]{mazurowski2024infinitely} to deduce that one of the following two alternatives holds:
\begin{itemize}
    \item[(i)] There is a smooth, closed, almost-embedded hypersurface $\Lambda$ with non-zero constant mean curvature such that $\bd \Omega = \Lambda$ and $V = \vert \Lambda\vert$. 
    \item[(ii)] The varifold $V$ is induced by a collection of disjoint, closed minimal hypersurfaces with integer multiplicities. Moreover, some subcollection of the minimal hypersurfaces bounds the region $\Omega$. 
\end{itemize}
Moreover, the first alternative holds when $H\neq 0$, and in this case $\Lambda$ has constant mean curvature $H$. The second alternative holds when $H = 0$. 

We claim that in case (ii) one has $\vol(\Omega) = \hv$. Indeed, this follows from the relation $\vert H_k\vert = c \vert f_{\eps_k}'(\vol(\Omega_k)-\hv)\vert$ and the fact that $H_k\to 0$ when alternative (ii) holds. Since the metric $g$ is assumed to be generic, no collection of minimal hypersurfaces in $M$ encloses half the volume of $M$. Therefore,  case (ii) cannot occur. 

Next, let us analyze case (i) in more detail. Observe that if $\vol(\Omega) \neq \hv$, then the relation $\vert H_k\vert = c \vert f_{\eps_k}'(\vol(\Omega_k)-\hv)\vert$ implies that $\vert H\vert = c$. Therefore, if $\vol(\Omega)\neq \hv$, the hypersurface $\Lambda$ has constant mean curvature $c$. The other possibility is that $\vol(\Omega) = \hv$, in which case $\Lambda$ is a hypersurface with constant mean curvature less than $c$ which encloses half the volume of $M$. 
Finally, note that 
\[
E_k(V_k,\bd \Omega_k) \ge \alpha_{p+1} - c \cdot \sup_{t\in \R} \big \vert f_{\eps_k}(t) - \vert t\vert\big \vert 
\]
and consequently 
\[
\|V_k\|(M) \ge \alpha_{p+1} - c\cdot \sup_{t\in \R} \big \vert f_{\eps_k}(t)-\vert t\vert\big \vert. 
\]
Letting $k\to \infty$, we obtain $\|V\|(M)\ge \alpha_{p+1}$. 

To summarize, we have now shown that for every $p\in \N$ for which $\alpha_p < \alpha_{p+1}$, there exists a closed, almost-embedded hypersurface $\Lambda$ in $M$ with $\area(\Lambda) \ge \alpha_{p+1}$ such that either 
\begin{itemize}
    \item[(i)] $\Lambda$ has constant mean curvature $c$, or 
    \item[(ii)] $\Lambda$ has constant mean curvature less than $c$ and encloses half the volume of $M$. 
\end{itemize}
This completes the proof of Theorem \ref{theorem:gap}.
\end{proof}

\bibliographystyle{plain}
\bibliography{bibliography}

\begin{thebibliography}{10}

\bibitem{almgren1962homotopy}
Frederick Almgren.
\newblock The homotopy groups of the integral cycle groups.
\newblock {\em Topology}, 1(4):257--299, 1962.

\bibitem{almgren1965theory}
Frederick Almgren.
\newblock The theory of varifolds.
\newblock {\em Mimeographed notes}, 1965.

\bibitem{arnold2004arnold}
Vladimir~I Arnold.
\newblock {\em Arnold's problems}.
\newblock Springer, 2004.

\bibitem{dey2022comparison}
Akashdeep Dey.
\newblock A comparison of the almgren--pitts and the allen--cahn min--max theory.
\newblock {\em Geometric and Functional Analysis}, 32(5):980--1040, 2022.

\bibitem{Dey23}
Akashdeep Dey.
\newblock Existence of multiple closed {CMC} hypersurfaces with small mean curvature.
\newblock {\em J. Differential Geom.}, 125(2):379--403, 2023.

\bibitem{gromov2006dimension}
Mikhael Gromov.
\newblock Dimension, non-linear spectra and width.
\newblock In {\em Geometric Aspects of Functional Analysis: Israel Seminar (GAFA) 1986--87}, pages 132--184. Springer, 2006.

\bibitem{liokumovich2018weyl}
Yevgeny Liokumovich, Fernando Marques, and Andr{\'e} Neves.
\newblock Weyl law for the volume spectrum.
\newblock {\em Annals of Mathematics}, 187(3):933--961, 2018.

\bibitem{marques2017existence}
Fernando~C Marques and Andr{\'e} Neves.
\newblock Existence of infinitely many minimal hypersurfaces in positive ricci curvature.
\newblock {\em Inventiones mathematicae}, 209(2):577--616, 2017.

\bibitem{marques2021morse}
Fernando~C Marques and Andr{\'e} Neves.
\newblock Morse index of multiplicity one min-max minimal hypersurfaces.
\newblock {\em Advances in Mathematics}, 378:107527, 2021.

\bibitem{Mazurowski22}
Liam Mazurowski.
\newblock C{MC} doublings of minimal surfaces via min-max.
\newblock {\em J. Geom. Anal.}, 32(3):Paper No. 104, 28, 2022.

\bibitem{mazurowski2023half}
Liam Mazurowski and Xin Zhou.
\newblock The half-volume spectrum of a manifold.
\newblock {\em arXiv preprint arXiv:2302.07722}, 2023.

\bibitem{mazurowski2024infinitely}
Liam Mazurowski and Xin Zhou.
\newblock Infinitely many half-volume constant mean curvature hypersurfaces via min-max theory.
\newblock {\em arXiv preprint arXiv:2405.00595}, 2024.

\bibitem{pacard2005foliations}
F~Pacard and R~Mazzeo.
\newblock Foliations by constant mean curvature tubes.
\newblock {\em Communications in analysis and geometry}, 13(4):633--670, 2005.

\bibitem{pacard2005constant}
Frank Pacard.
\newblock Constant mean curvature hypersurfaces in riemannian manifolds.
\newblock {\em Riv. Mat. Univ. Parma (7)}, 4:141--162, 2005.

\bibitem{pacard2009constant}
Frank Pacard and Xingwang Xu.
\newblock Constant mean curvature spheres in riemannian manifolds.
\newblock {\em manuscripta mathematica}, 128:275--295, 2009.

\bibitem{pitts2014existence}
Jon~T Pitts.
\newblock {\em Existence and regularity of minimal surfaces on Riemannian manifolds.(MN-27)}, volume~27.
\newblock Princeton University Press, 2014.

\bibitem{schneider2011closed}
Matthias Schneider.
\newblock Closed magnetic geodesics on {$S^2$}.
\newblock {\em Journal of Differential Geometry}, 87(2):343--388, 2011.

\bibitem{simon1983lectures}
Leon Simon.
\newblock {\em Lectures on geometric measure theory}, volume~3 of {\em Proceedings of the Centre for Mathematical Analysis, Australian National University}.
\newblock Australian National University, Centre for Mathematical Analysis, Canberra, 1983.

\bibitem{wang2023existence}
Zhichao Wang and Xin Zhou.
\newblock Existence of four minimal spheres in {$S^3$} with a bumpy metric.
\newblock {\em arXiv preprint arXiv:2305.08755}, 2023.

\bibitem{ye1991foliation}
Rugang Ye.
\newblock Foliation by constant mean curvature spheres.
\newblock {\em Pacific Journal of Mathematics}, 147(2):381--396, 1991.

\bibitem{zhou2020multiplicity}
Xin Zhou.
\newblock On the multiplicity one conjecture in min-max theory.
\newblock {\em Annals of Mathematics}, 192(3):767--820, 2020.

\bibitem{Zhou-ICM22}
Xin Zhou.
\newblock Mean curvature and variational theory.
\newblock In {\em I{CM}---{I}nternational {C}ongress of {M}athematicians. {V}ol. {IV}. {S}ections 5--8}, pages 2696--2717. EMS Press, Berlin, [2023] \copyright 2023.

\bibitem{zhou2019min}
Xin Zhou and Jonathan~J Zhu.
\newblock Min--max theory for constant mean curvature hypersurfaces.
\newblock {\em Inventiones mathematicae}, 218:441--490, 2019.

\end{thebibliography}

\end{document}